\documentclass[
final
]{dmtcs-episciences}

\usepackage[utf8]{inputenc}
\usepackage{subfigure}

\usepackage[round]{natbib}

\usepackage{amsmath,amsthm,amssymb,amsfonts, latexsym,MnSymbol}
\usepackage{enumerate}
\usepackage{verbatim}
\usepackage{multicol}
\usepackage{graphicx}
\usepackage{epsfig}
\usepackage{color}
\usepackage{bezier}

\usepackage{mathtools}
\usepackage{array,booktabs}
\usepackage{nomencl} 

\newtheorem{thm}{Theorem}[section]
\newtheorem{cor}[thm]{Corollary}
\newtheorem{lem}[thm]{Lemma}
\newtheorem{prop}[thm]{Proposition}
\newtheorem{claim}[thm]{Claim}

\theoremstyle{definition}
\newtheorem{definition}[thm]{Definition}
\newtheorem{ex}[thm]{Example}
\newtheorem{exs}[thm]{Examples}
\theoremstyle{remark}
\newtheorem{rem}[thm]{Remark}
\newtheorem{problem}{Problem}[section]
\newtheorem{subclaim}[thm]{Subclaim}

\newtheorem{question}[problem]{Question}

\DeclarePairedDelimiter{\card}{\lvert}{\rvert} 

\DeclareMathOperator{\rank}{rank} 
\DeclareMathOperator{\minrank}{minrank} 

\DeclareMathOperator{\Span}{span}
\DeclareMathOperator{\ddgeo}{dim_{geom}} 

\DeclareMathOperator{\ddBool}{dim_{Bool}}
\DeclareMathOperator{\ddsymp}{dim_{symp}}
\DeclareMathOperator{\ind}{ind_{2}}
\DeclareMathOperator{\ddinn}{dim_{inn}}

\newcommand{\NN}{{\mathbb N}}

\definecolor{darkblue}{rgb}{0.0,0.0,0.3}
\definecolor{darkmagneta}{rgb}{0.5,0.0,0.5}
\definecolor{darkred}{rgb}{0.5,0.0,0.0}

\author{Maurice Pouzet\affiliationmark{1,2}
  \and Hamza Si Kaddour\affiliationmark{1}
  \and Bhalchandra D. Thatte\affiliationmark{3}\thanks{Supported by CAPES Brazil (Processo: 88887.364676/2019-00);
  the stay of this author was supported by LABEX MILYON (ANR-10-LABX-0070) of
  Universit\'e de Lyon within the program "Investissements d'Avenir
  (ANR-11-IDEX-0007)" operated by the French National Research Agency (ANR)}}
\title[On the Boolean dimension of a graph and other related
  parameters]{On the Boolean dimension of a graph \\ and other related
  parameters}
\affiliation{
  Univ. Lyon, Universit\'e Claude-Bernard Lyon1, CNRS UMR 5208, Institut Camille Jordan, 43, Bd. du 11 Novembre 1918, 69622
Villeurbanne, France \\
  Department of Mathematics and Statistics, University of Calgary, Calgary, Alberta, Canada\\
  Departamento de
  Matematica, Universidade Federal de Minas Gerais (UFMG), Av. Antonio Carlos,
  6627, Caixa Postal 702, Região Pampulha, Belo Horizonte - MG, CEP: 31270-901,
  Brasil}
\received{2021-05-03}
\revised{2022-04-04}
\accepted{2022-08-02}

\begin{document}
\publicationdetails{23}{2022}{2}{5}{7437}
\maketitle
\begin{abstract}
  We present the Boolean dimension of a graph, we relate it with
  the notions of inner, geometric and symplectic dimensions, and with the rank
  and minrank of a graph. We obtain an exact formula for the Boolean dimension
  of a tree in terms of a certain star decomposition. We relate the Boolean
  dimension with the inversion index of a tournament.
\end{abstract}

%
%



\section{Presentation and preliminaries}

We define the notion of Boolean dimension of a graph, as it appears in
\citet{bbbp2010} (see also \citep{belkhechine,bbbp2012}). We present the
notions of geometric and symplectic dimensions, and the rank and minrank of a
graph, which have been considered earlier.  When finite, the Boolean dimension
corresponds to the inner dimension; it plays an intermediate role between the
geometric and symplectic dimensions, and does not seem to have been considered
earlier. The notion of Boolean dimension was introduced in order to study
tournaments and their reduction to acyclic tournaments by means of
inversions. The key concept is the inversion index of a tournament
\citep{belkhechine, bbbp2010, bbbp2012} presented in
Section~\ref{section:inversionindex}. Our main results are an exact formula for
the Boolean dimension of a tree in terms of a certain star decomposition
(Theorem~\ref{thm:trees}) and the computation of the inversion index of an
acyclic sum of $3$-cycles (Theorem~\ref{thm:threecycle}).

Notations in this paper are quite elementary. The \emph{diagonal} of a set $X$
is the set $\Delta_X:= \{(x,x): x\in X\}$. We denote by $\powerset (X)$ the
collection of subsets of $X$, by $X^m$ the set of $m$-tuples $(x_1, \dots, x_m)$
of elements in $X$, by $[X]^m$ the $m$-element subsets of $X$, and by
$[X]^{<\omega}$ the collection of finite subsets of $X$. The cardinality of $X$
is denoted by $\card{X}$. We denote by $\aleph_0$ the first infinite cardinal,
by $\aleph_1$ the first uncountable cardinal, and by $\omega_1$ the first
uncountable ordinal. A cardinal $\kappa$ is \emph{regular} if no set $X$ of
cardinal $\kappa$ can be divided in strictly less than $\kappa$ subsets, all of
cardinality strictly less than $\kappa$. If $\kappa$ denotes a cardinal,
$2^{\kappa}$ is the cardinality of the power set $\powerset (X)$ of any set $X$
of cardinality $\kappa$. If $\kappa$ is an infinite cardinal, we set
$\log_2(\kappa)$ for the least cardinal $\mu$ such that $\kappa \leq
2^{\mu}$. We note that for an uncountable cardinal $\kappa$ the equality
$\log_2(2^{\kappa})= \kappa$ may require some set theoretical axioms, such as
the Generalized Continuum Hypothesis (GCH). If $\kappa$ is an integer, we use
$\log_2(\kappa)$ in the ordinary sense, hence the least integer $\mu$ such that
$\kappa \leq 2^{\mu} $ is $\lceil \log_2\kappa \rceil$.  We refer the reader to
\citet{jech} and \citet{kunen} for further background about axioms of set theory
if needed.

The graphs we consider are undirected and have no loops. They do not need to be
finite, but our main results are for finite graphs. A \emph{graph} is a pair
$(V, E)$ where $E$ is a subset of $[V]^2$, the set of $2$-element subsets of
$V$. Elements of $V$ are the \emph{vertices} and elements of $E$ are the
\emph{edges}. Given a graph $G$, we denote by $V(G)$ its vertex set and by
$E(G)$ its edge set. For $u,v\in V(G)$, we write $u\sim v$ and say that $u$ and
$v$ are \emph{adjacent} if there is an edge joining $u$ and $v$.  The
\emph{neighbourhood} of a vertex $u$ in $G$ is the set $N_G(u)$ of vertices
adjacent to $u$. The \emph{degree} $d_G(u)$ of a vertex $u$ is the cardinality
of $N_G(u)$. If $X$ is a subset of $V(G)$, the \emph{subgraph of $G$ induced by
  $X$} is $G_{\restriction X}:= (X, E\cap [X]^2)$. A \emph {clique} in a graph
$G$ is a set $X$ of vertices such that any two distinct vertices in $X$ are
adjacent. If $X$ is a subset of a set $V$, we set $K^V_{X}:= (V, [X]^2)$; we say
also that this graph is a clique.

\subsection{The Boolean sum of graphs and the Boolean dimension of a graph}
Let $(G_i)_{i\in I}$ be a family of graphs, all with the same vertex set
$V$. The \emph{Boolean sum} of this family is the graph, denoted by
$\dot {+} (G_i)_{i\in I}$, with vertex set $V$ such that an unordered pair
$e:= \{x,y\}$ of distinct elements of $V$ is an edge if and only if it belongs
to a finite and odd number of $E(G_i)$. If the family consists of two elements,
say $(G_i)_{i\in \{0,1\}}$ we denote this sum by $G_{0}\dot{+} G_{1}$. This is
an associative operation (but, beware, infinite sums are not associative). If
each $E(G_i)$ is the set of edges of some clique $C_i$, we say (a bit
improperly) that $\dot {+} (G_i)_{i\in I}$ is a sum of cliques. We define the
\emph{ Boolean dimension} of a graph $G$, which we denote by $\ddBool(G)$, as
the least cardinal $\kappa$ such that $G$ is a Boolean sum of $\kappa$
cliques. In all, $\ddBool (G)=\kappa$ if there is a family of $\kappa$ subsets
$(C_i)_{i\in I}$ of $V(G)$, and not less, such that an unordered pair
$e:= \{x,y\}$ of distinct elements is an edge of $G$ if and only if it is
included in a finite and odd number of $C_i$'s.

A \emph{Boolean representation} of a graph $G$ in a set $E$ is a map
$f: V(G)\rightarrow \powerset (E)$ such that an unordered pair $e:= \{x,y\}$ of
distinct elements is an edge of $G$ if and only if the intersection
$f(x)\cap f(y)$ is finite and has an odd number of elements.

\begin{ex}\label{example:1}
  Let $G$ be a graph. For a vertex $x\in V(G)$, let
  $E_G(x):=\{e\in E(G): x\in e\}$.  Set $E:= E(G)$. Then the map
  $f:V(G)\rightarrow \powerset (E)$ defined by $f(x):=E_G(x)$ is a Boolean
  representation. Indeed, for every $2$-element subset $e:=\{x , y\}$ of $V(G)$,
  the intersection $f(x)\cap f(y)$ has one element if and only if $e\in E(G)$, otherwise it
  is empty.
\end{ex}

The following result is immediate, still it has some importance. 

\begin{prop} A graph $G$ is a Boolean sum of $\kappa$ cliques if and only if $G$
  has a Boolean representation in a set of cardinality $\kappa$.
\end{prop} 
\begin{proof} If $G$ is the Boolean sum of a family $(C_i)_{i\in I}$ of $\kappa$
  cliques, then let $f: V(G) \rightarrow \powerset (I)$ defined by setting
  $f(x):= \{i\in I: x\in C_i\}$. This defines a Boolean representation in
  $I$. Conversely, if $f: V(G) \rightarrow \powerset (E)$ is a Boolean
  representation in a set $E$, then set $C_i: = \{ x\in V(G): i\in f(x)\}$ for
  $i\in E$. Then $G$ is the Boolean sum of the family $(G_i)_{i\in E}$, where
  $G_i:= (V(G), [C_i]^2)$ for each $i\in E$.
\end{proof}

We note that the Boolean dimension of a graph and of the graph obtained by
removing some isolated vertices are the same. Hence $\ddBool (G)=1$ if and only
if it is of the form $G= K^V_X$ with $\vert X\vert\geq2$. Since every graph
$G:= (V, E)$ can be viewed as the Boolean sum of its edges, the Boolean
dimension of $G$ is always defined, and is at most the number of edges, that is,
at most the cardinality $\card{[V]^2}$ of $ [V]^2$. If $V$ is infinite, then
$\card{[V]^2}= \card{V}$; hence $\ddBool(G)\leq \vert V\vert $ (but see
Question~\ref{question1.1} below). If $V$ is finite, with $n$ elements, then
$\ddBool(G)\leq n-1$ (see \cite{bbbp2010}. By induction on $n$: pick $x\in V$ and
observe that 
$G= (V, E\setminus \{e \in E : x\in e\})\dot{+} K^V_{N_{G}(x)}\dot{+}
K^V_{N_{G}(x)\cup \{x\}}$). In fact, paths on $n$ vertices are the only
$n$-vertex graphs with Boolean dimension $n-1$, see Theorem~\ref{thm:paths}, a
result that requires some ingredients developed below.

Recall that a \emph{module} in a graph $G$ is any subset $A$ of $V(G)$ such that
for every $a,a'\in A$ and $b\in V(G)\setminus A$, we have $a\sim b$ if and only
if $a'\sim b$. A \emph{duo} is any two-element module (e.g., see
\citet{courcelle-delhomme} for an account of the modular decomposition of
graphs).
\begin{lem}\label{lem:onetoone}
  If a graph $G$ has no duo then every Boolean representation is one to one. In
  particular, $\ddBool(G)\geq \log_2(\card{V(G)})$.
\end{lem}
\begin{proof}
  Observe that if $f$ is a representation and $v$ is in the range of $f$, then
  $f^{-1} (v)$ is a module and this module is either a clique or an independent
  set.
\end{proof}

\begin{question}\label{question1.1}
  The inequality in Lemma~\ref{lem:onetoone} may be strict if $V(G)$ is finite.
  Does $\ddBool(G)\leq \log_2(\vert V(G)\vert)$ when $V(G)$ is infinite? The
  answer may depend on some set theoretical hypothesis (see
  Example~\ref{exs:infiniteddbool}).  But we do not known if the Boolean
  dimension of every graph on at most a continuum of vertices is at most
  countable.  Same question may be considered for trees.
\end{question}

Let $E$ be a set; denote by $O(E)^{\neg \bot}$ the graph whose vertices are the
subsets of $E$, two vertices $X$ and $Y$ being linked by an edge if they are
distinct and their intersection is finite and odd. If $\kappa$ is a cardinal, we
set $O(\kappa)^{\neg \bot}$ for any graph isomorphic to $O(E)^{\neg \bot}$,
where $E$ is a set of cardinality $\kappa$.

\begin{thm}\label{thm:booleandim} A graph $G$ with no duo has Boolean dimension
  at most $\kappa$ if and only if it is embeddable in $O(\kappa)^{\neg
    \bot}$. The Boolean dimension of $O(\kappa)^{\neg \bot}$ is at most
  $\kappa$. It is equal to $\kappa$ if $\kappa \geq 2$ and $\kappa$ is at most
  countable, or if $\kappa$ is uncountable and (GCH) holds.
\end{thm}  
\begin{proof}
  If there is an embedding $f$ from $G$ in a graph of the form
  $O(E)^{\neg \bot}$, then $f$ is a Boolean representation of $G$, hence
  $\ddBool (G)\leq \vert E\vert$. Conversely, if $G$ has no duo and has a
  Boolean representation $f$ in a set $E$ then, by Lemma~\ref{lem:onetoone}, $f$
  is an embedding of $G$ in $O(E)^{\neg \bot}$.  Let $E$ be a set of cardinality
  $\kappa$. For each $X\in V(O(E)^{\neg \bot})=\powerset (E)$ set $f(X):= X$ viewed as a subset
  of $\powerset (E)$. The map $f$ is a Boolean representation, hence
  $\ddBool O(E)^{\neg \bot} \leq \kappa$. Alternatively, set
  $C_i:= \{X\in \powerset (E): i\in X\}$ for each $i\in E$. Then
  $O(E)^{\neg \bot}$ is the Boolean sum of the $[C_i]^2, i \in E$.  If
  $\kappa = 2$, a simple inspection shows that the Boolean dimension of
  $O(E)^{\neg \bot}$ is $\kappa$. If $\kappa \geq 3$ then $O(E)^{\neg \bot}$ has
  no duo. This relies on the following claim.
 \begin{claim}
   If $A, B$ are two distinct subsets of $E$, then there is a subset $C$ of $E$,
   distinct from $A$ and $B$, with at most two elements such that the
   cardinalities of the sets $A\cap C$ and $B\cap C$ cannot have the same
   parity.
 \end{claim}
 Indeed, we may suppose that $A \not \subseteq B$. Pick $x\in A\setminus B$. If
 $\vert A\vert >1$, then set $C:= \{x\}$. If not, then $A= \{x\}$. In this case,
 either $B$ is empty and $C:= \{x, y\}$, with $y\not =x$ will do, or $B$ is
 nonempty, in which case, we may set $C:= \{y\}$, where $y\in B$ if
 $\vert B\vert >1$, or $C:= \{y, z\}$, where $B=\{y\}$ and
 $z\in E\setminus (A\cup B)$.
 
 Since $O(E)^{\neg \bot}$ has no duo, Lemma~\ref {lem:onetoone} ensures that
 $\ddBool(O(E)^{\neg \bot})\geq \log_2(\vert V(O(E))\vert )=
 \log_2(2^{\kappa})$. If $\kappa$ is at most countable, or $\kappa$ is
 uncountable and (GCH) holds, then this last quantity is $\kappa$. This
 completes the proof of the theorem.  \end{proof}

We can obtain the same conclusion with a weaker hypothesis than (GCH).

\begin{lem}\label{lem:weaker} Let $\kappa$ be an infinite cardinal. If
  $\mu^{\omega}< \kappa$ for every $\mu<\kappa$, then
  $\ddBool(O(\kappa)^{\neg \bot})=\kappa$.
\end{lem} 

\begin{proof} 
The proof relies on the following claim, which is of independent interest. 
\begin{claim}\label{claim:clique}

  Let $\mu^{\omega}$ be the cardinality of the set of countable subsets of an
  infinite cardinal $\mu$. Then the cliques in $O(\mu)^{\neg \bot}$ have
  cardinality at most $\mu^{\omega}$.
\end{claim}
The proof relies on a property of almost disjoint families. Let us recall that
an \emph{almost disjoint family} is a family
$\mathcal A:= (A_{\alpha})_{\alpha\in I}$ of sets such that the intersection
$A_{\alpha} \cap A_{\beta}$ is finite for $\alpha\not =\beta$. Note that if
$\mathcal C$ is a clique in $O(\mu)^{\neg \bot}$, then for every pair of
distinct sets $X, Y$ in $\mathcal C$, the intersection $X\cap Y$ is finite and
its cardinality is odd. Hence, $\mathcal C$ is an almost disjoint family.

To prove our claim it suffices to prove the following claim, well known by set
theorists.
\begin{claim}\label{deltasystem}
  There is no almost disjoint family of more that $\mu^\omega$ subsets of an
  infinite set of cardinality $\mu$.
\end{claim}

\noindent{\bf Proof of Claim~\ref{deltasystem}.} Suppose that such a family
$\mathcal A:=(A_{\alpha})_{\alpha\in I}$ exists, with
$\vert I\vert >\mu^{\omega}$. Since $\mu^{<\omega} = \mu$, we may suppose that
each $A_{\alpha}$ is infinite and then select a countable subset $B_{\alpha}$ of
$A_{\alpha}$. The family $\mathcal B:= (B_{\alpha})_{\alpha\in I}$ is almost
disjoint, but since $\vert I\vert >\mu^{\omega}$, there are
$\alpha \not = \beta$ such that $B_{\alpha} = B_{\beta}$, hence
$B_{\alpha}\cap B_{\beta}$ is infinite, contradicting the fact that $\mathcal B$
is an almost disjoint family.  \hfill $\Box$

Now the proof of the lemma goes as follows. Suppose that
$\ddBool(O(\kappa)^{\neg \bot}= \mu< \kappa$. Then there is an embedding from
the graph $O(\kappa)^{\neg \bot}$ into the graph $O(\mu)^{\neg \bot}$.
Trivially, $O(\kappa)^{\neg \bot}$ contains cliques of cardinality at least
$\kappa$. Hence $O(\mu)^{\neg \bot}$ too. But since $\mu^{\omega}< \kappa$,
Claim~\ref{claim:clique} says that this is impossible. Thus
$\ddBool(O(\kappa)^{\neg \bot})= \kappa$.
\end{proof}

We thank \citet{avraham}  for providing Claim~\ref{deltasystem}. 

\begin{exs}\label{exs:infiniteddbool}
  For a simple illustration of Lemma~\ref{lem:weaker}, take
  $\kappa= (2^{\aleph_0})^+$ the successor of $2^{\aleph_0}$. For an example,
  negating (GCH), suppose $\omega_1= 2^{\aleph_0}$, $\kappa= \omega_2$,
  $\omega_3= 2^{\omega_1} = 2^{\omega_2}$. In this case,
  $\ddBool(O(\kappa)^{\neg \bot})= \kappa$ and
  $\log_2(2^{\kappa})= \omega_1< \kappa$.
\end{exs}

\begin{question} Does the equality $\ddBool(O(\kappa)^{\neg \bot})=\kappa$ hold
  without any set theoretical hypothesis?
\end{question}

\begin{rem} Theorem~\ref{thm:booleandim} asserts that $O(\kappa)^{\neg \bot}$ is
  universal among graphs with no duo of Boolean dimension at most $\kappa$ (that
  is embeds all graphs with no duo of dimension at most $\kappa$), but we do not
  know which graphs on at most $2^{\kappa}$ vertices embed in
  $O(\kappa)^{\neg \bot}$.\end{rem}

In contrast with Claim~\ref{claim:clique} we have:
\begin{lem} \label{lem:infinite trees} For an infinite cardinal $\kappa$, the
  graph $O(\kappa)^{\neg \bot}$ embeds a graph made of $2^{\kappa}$ disjoint
  edges. It embeds also some trees made of $2^{\kappa}$ vertices.
\end{lem} 
\begin{proof}
  Let $G$ be the graph made of $2^{\kappa}$ disjoint edges
  $\{a_{\alpha}, b_{\alpha}\}$ with $\alpha\in 2^{\kappa}$. We show that $G$ is
  isomorphic to an induced subgraph of $O(E)^{\neg \bot}$, where $E$ is the set
  $ [\kappa]^{<\omega}$ of finite subsets of $\kappa$, augmented of an extra
  element $r$. Since $\vert E\vert = \kappa$, this proves our first
  statement. For the purpose of the proof, select $2^{\kappa}$ subsets
  $X_{\alpha}$ of $\kappa$ which are pairwise incomparable with respect to
  inclusion and contain an infinite subset $X$. For each $\alpha\in 2^{\kappa}$,
  let $A_{\alpha}:= [X_{\alpha}]^{<\omega}\cup \{r\}$ and
  $B_{\alpha}:= E\setminus [X_{\alpha}]^{<\omega}$.  We claim that the subgraph $H$
  of $O(E)^{\neg \bot}$ induced by 
  $\{A_{\alpha} : \alpha \in 2^{\kappa}\} \cup \{B_{\alpha}: \alpha \in 2^{\kappa}\}$ is a direct sum of
  the edges $\{ A_{\alpha} , B_{\alpha} \}$ for $\alpha\in 2^{\kappa}$. That
  $A_{\alpha}$ and $B_{\alpha}$ form an edge is obvious: their intersection is
  the one element set $\{r\}$.  Now, let $\alpha \not = \beta$. We claim that
  the three intersections $A_{\alpha}\cap A_{\beta}$,
  $A_{\alpha} \cap B_{\beta}$ and $B_{\alpha} \cap B_{\beta}$ are all
  infinite. For the first one, this is obvious (it contains
  $[X_{\alpha} \cap X_{\beta}]^{< \omega}$), for the next two, use the fact that
  the $A_{\alpha}$ are up-directed with respect to inclusion, hence the
  difference $A_{\alpha} \setminus A_{\beta}$ is cofinal in $A_{\alpha}$, thus
  must be infinite, and the union $A_{\alpha} \cup A_{\beta}$ cannot cover
  $[\kappa]^{<\omega}$, hence its complement is infinite.  It follows that the
  graph $H$ contains no other edges than the pairs
  $\{ A_{\alpha} , B_{\alpha} \}$'s. This proves that $H$ is isomorphic to $G$,
  and yields our first statement. For the second statement, add
  $R:= [X]^{<\omega} \cup \{r\}$ to the set of vertices of $H$. We get a
  tree. Indeed, for each $\alpha$, the vertices $R$ and $A_{\alpha}$ do not form
  an edge in $O(E)^{\neg \bot}$ (indeed, $R\cap A_{\alpha}= [X]^{<\omega}$ hence
  is infinite), while for each $\beta$, the vertices $R$ and $A_{\beta}$ form an
  edge (since $R\cap A_{\beta}= \{r\}$).
\end{proof}

For infinite graphs with finite Boolean dimension, a straightforward application
of Tychonoff's theorem yields the following result.

\begin{thm} Let $n\in \NN$. For every graph $G$, $\ddBool(G)\leq n$ if and only
  if $\ddBool (G_{\restriction X} )\leq n$ for every finite subset $X$ of
  $V(G)$.
\end{thm}
\begin{proof} Suppose that the second condition holds. For every finite subset
  $X$ of $V(G)$ let $U_X$ be the set of maps $f$ from $V(G)$ into the powerset
  $K:= \powerset (\{1, \dots, n\})$ such that the restriction
  $f_{\restriction X}$ is a Boolean representation of $G_{\restriction X}$ in
  $\{1, \dots, n\}$. Each such set $U_X$ is nonempty and closed in the set
  $K^{V(G)}$ equipped with the product topology, the set $K$ being equipped with the discrete topology. Every finite intersection
  $U_{X_1}\cap \dots \cap U_{X_\ell}$ contains $U_{X_1\cup \dots\cup X_{\ell}}$
  hence is nonempty. The compactness of $K^{V(G)}$ ensures that the intersection
  of all of those sets is nonempty. Any map in this intersection is a Boolean
  representation of $G$.
\end{proof}

Examples of graphs with finite Boolean dimension are given at the end of the
next subsection.

\subsection{Geometric notions of dimensions of graphs} We introduce three
notions of dimensions: geometric, inner, and symplectic, all based on bilinear
forms. We prove that if the Boolean dimension of a graph is finite, then it
coincides with the inner dimension, and either these dimensions minus $1$
coincide with the geometric and the symplectic dimension, or they coincide with
the geometric dimension, the symplectic being possibly larger
(Theorem~\ref{thm:threedimension}). We note before all that in general, the
Boolean dimension is not based on a bilinear form. It uses the map
$\varphi: \powerset (E)\rightarrow 2:= \{0, 1\}$ defined by setting
$\varphi(X,Y):= 1$ if $\vert X\cap Y\vert $ is finite and odd and $0$
otherwise. But except when $E$ is finite, it is not bilinear on $\powerset (E)$
equipped with the symmetric difference.

Let $\mathbb {F}$ be a field, and let $U$ be a vector space over $\mathbb{F}$,
and let $\varphi$ be a bilinear form over $U$. We recall that this form is {\em
  symmetric} if $\varphi(x,y) = \varphi(y,x)$ for all $x,y\in U$. Two vectors
$x,y$ are \emph{orthogonal} if $\varphi(x,y)=0$. A vector $x\in U$ is \emph
{isotropic} if $\varphi(x,x)=0$.  The \emph{orthogonal} of a subset $X$ of $U$
is the subspace
$X^{\bot}:=\{ y\in U: \varphi (x, y)=0 \; \text {for all}\; x\in X\}$. We set
$x^{\bot}$ instead of $\{x\}^{\bot}$. We recall that $\varphi$ is
\emph{degenerate} if there is some $x\in U\setminus \{0\}$ such that
$\varphi(x,y)=0$ for all $y\in U$.  The form $\varphi$ is said to be
\emph{alternating} if each $x \in U$ is isotropic, in which case $(U,\varphi)$
is called a {\em symplectic space}. The form $\varphi$ is an \emph{inner form}
or a \emph{scalar product} if $U$ has an \emph{orthonormal basis} (made of
non-isotropic and pairwise othogonal vectors).

\begin{definition} Let $U$ be a vector space equipped with a
  symmetric bilinear form $\varphi$. Let $G$ be a graph. We say that a map
  $f \colon V(G) \to U$ is a {\em geometric representation} of $G$ in
  $(U,\varphi)$ if for all $u,v \in V(G), u \neq v$, we have $u \sim v$ if and
  only if $\varphi(f(u),f(v)) \neq 0$.  The {\em geometric dimension} of $G$,
  denoted by $\ddgeo(G)$, is the least cardinal $\kappa$ for which there exists
  a geometric representation of $G$ in a vector space $U$ of dimension $\kappa$
  equipped with a symmetric bilinear form $\varphi$. The {\em symplectic
    dimension} of $G$, denoted by $\ddsymp(G)$, is the least cardinal $\kappa $
  for which there exists a symplectic space $(U,\varphi)$ in which $G$ has a
  geometric representation. The {\em inner dimension} of $G$, denoted by
  $\ddinn(G)$, is the least cardinal $\kappa$ for which $G$ has a geometric
  representation in a vector space of dimension $\kappa$ equipped with a scalar
  product.
\end{definition}

The notions of geometric and symplectic dimension were considered by several
authors, for example, \citet{garzon1987, godsil-royle2001}. There is an extensive
literature about this subject (e.g. \citet{fallat-hogben2007,grout}),
and notably the role of the field.  But apparently, the Boolean dimension was
not considered.

Except in subsection~\ref{subsection:dimensionrank}, we consider these notions
only for the $2$-element field $\mathbb{F}_2$, identified with the set
$\{0,1\}$.  If $U$ has finite dimension, say $k$, we identify it with
$\mathbb{F}_2^k$, the set of all $k$-tuples over $\{0,1\}$; the basis
$(e_i)_{i:=1,\dots, k}$, where $e_i$ is the $0$-$1$-vector with a $1$ in the
$i$-th position and $0$ elsewhere, is orthonormal; the scalar product of two
vectors $x:= (x_1, \dots, x_k)$ and $y:= (y_1, \dots, y_k)$ of $\mathbb{F}_2^k$
is then $\langle x\mid y\rangle:= x_1y_1+ \dots +x_ky_k$.  We recall the following
dichotomy result.

\begin{thm}\label{thm:nondegenerate}
  A nondegenerate bilinear symmetric form $\varphi$ on a finite $k$-dimensional
  space $U$ over the two-element field $\mathbb{F}_2$ falls into two
  types. Either $\varphi$ is non-alternating and $(U, \varphi)$ is isomorphic to
  $(\mathbb F_2^k, |)$ with the scalar product, or $\varphi$ is alternating, $k$
  is even, and $(U, \varphi)$ is isomorphic to the symplectic space
  $H(k):=(\mathbf{1}^ {\neg \bot}, |_{\restriction \mathbf{1}^ {\neg \bot}})$,
  where $\mathbf{1}^ {\neg \bot}$ is the orthogonal of
  $\mathbf 1:= (1, \dots, 1)$ with respect to the scalar product $|$ on
  $\mathbb F_2^{k+1}$.
\end{thm} 

For reader's convenience, we give a proof. The proof, suggested by
Christian Delhomm\'e, is based on two results exposed in Algebra, Vol. 3,  of
 \citet{cohn}.  Let $(U,\varphi)$ be as stated in the above
theorem. Case 1: $\varphi$ is not
symplectic, that is $\varphi (x,x)\neq 0$ for some vector $x$. We apply
Proposition 7.1 page 344 of \citet{cohn}, namely: \emph{If $U$ is a vector space
  of characteristic $2$ and $\varphi$ is a symmetric bilinear form which is not
  alternating, then $U$ has a orthogonal basis}. Since $\varphi$ is 
  nondegenerate and the field if $\mathbb F_2$, any orthogonal basis is orthonormal, hence $\varphi$ is a scalar
product. Case 2: $\varphi$ is symplectic. In this case, Lemma 5.1, p.331 of
\citet{cohn} asserts in particular  that: \emph{Every symplectic space,
  (that is a space equipped with a bilinear symmetric form which is 
  nondegenerate and alternating) on an arbitrary field is a sum of hyperbolic
  planes}. Thus $k$ is even and in our case $U$ is isomorphic to any symplectic
space with the same dimension, in particular to $H(k)$.


When dealing with these notions of dimension, we may always consider
nondegenerate forms, hence in the case of finite dimensional representation,
Theorem~\ref{thm:nondegenerate} applies. In fact Lemma~\ref{lem:onetoone} and
Theorem~\ref {thm:booleandim} extend.

Let $U$ be a vector space over $\mathbb F_2$ and $\varphi$ a symmetric bilinear
form defined on $U$ with values in $\mathbb F_2$. Let $O^{\neg \bot}_{\varphi}$
be the graph of the non-orthogonality relation on $U$, that is, the graph whose
edges are the pairs of distinct elements $x$ and $y$ such that $\varphi(x,y)=1$.
If $k$ is an integer, then we denote by $O^{\neg \bot}_{\mathbb F_2}(k)$ the
graph on $\mathbb F_{2}^k$ of the non-orthogonality relation associated with the
inner product $|$. Similarly, for $k$ even, let $O^{\neg \bot}_H(k)$ be the
graph on $H(k)$, the orthogonal of $\mathbf 1:= (1, \dots, 1)$ with respect to
the scalar product $|$ on $(\mathbb F_2)^{k+1}$, equipped with the symplectic
form induced by the scalar product.

\begin{lem}\label{lem:dimensionphi} If $\dim (U)$, the dimension of the vector
  space $U$, is at least $3$, then the graph $O^{\neg \bot}_{\varphi}$ has no
  duo if and only if $\varphi$ is nondegenerate. Hence,
  $\ddgeo(O^{\neg \bot}_{\varphi})=\dim(U)$ when $\varphi$ is nondegenerate.
 \end{lem}
 \begin{proof}
   Suppose that $\varphi$ is degenerate. Pick a nonzero element $a$ in the
   kernel of $\varphi$. Then, as it is easy to check, the $2$-element set
   $\{0, a\}$ is a module of $O^{\neg \bot}_{\varphi}$. Conversely, let
   $\{a, b\}$ be a duo of $O^{\neg \bot}_{\varphi}$. We claim that $c:=a+b$
   belongs to the kernel of $\varphi$, that is $\varphi (x, c)=0$ for every
   $x\in U$.  Indeed, if $x\not \in \{a,b\}$, then $\varphi(x, a)=\varphi(x,b)$,
   hence $\varphi (x,c)=0$ since $\{a,b\}$ is a module. If $x\in \{a,b\}$ (e.g.
   $x:=a$), then since $\dim(U)\geq 3$, we may pick some
   $z\not \in \Span\{a,b\}:= \{0, a,b, a+b\}$, hence $\varphi(z, c)=0$. Since
   $z+a\not \in \{a,b\}$, $\varphi(z+a, c)=0$. It follows that $\varphi(a,c)=0$,
   proving our claim.  According to Lemma~\ref{lem:onetoone}, every
   representation of $O^{\neg \bot}_{\varphi}$ is one to one; since the identity
   map is a representation, we have $\ddgeo( O^{\neg \bot}_{\varphi}) =\dim(U)$.
 \end{proof} 

 We give below an existential result. The proof of the second item is based on
 the $\Delta$-system lemma (see \citep{kunen,rinot} for an elementary
 proof) that we recall now.

 \begin{lem}\label {delatsystemlemma}
   Suppose that $\kappa$ is a regular uncountable cardinal, and
   $\mathcal A:=(A_{\alpha})_{\alpha\in \kappa}$ is a family of finite sets.
   Then there exist a subfamily $\mathcal B:=(A_{\alpha})_{\alpha\in K}$, where
   the cardinality of $K$ is $\kappa$, and a finite set $R$ such that
   $A_{\alpha} \cap A_{\beta}= R$ for all distinct $\alpha, \beta \in K$.
\end{lem}

\begin{thm}\label{thm-geometric}\begin{enumerate} 
  \item Every graph has a symplectic dimension, and hence, it has a geometric
    one. However:
  \item not every graph has an inner dimension, e.g., a graph with $\kappa$
    vertices, with $\kappa$ regular, and no clique and no independent set of
    $\kappa$ vertices, does not have an inner representation; on an other hand:
  \item every locally finite graph has an inner dimension. 
\end{enumerate}
\end{thm} 

\begin{proof}
  \begin{enumerate}
  \item Let $G$ be a graph, and $\kappa:= \vert V(G)\vert$. Let $U$ be a vector
  space over $\mathbb{F}_2$ with dimension $\kappa$ (e.g.,
  $U:= \mathbb{F}_2^{[V(G)]}$, the set of maps
  $f: V(G) \rightarrow \mathbb{F}_2$ which are $0$ almost everywhere).  Define a
  symplectic form $\varphi$ on a basis $\mathcal B:= \{ b_v: v\in V(G)\}$ of $U$
  indexed by the elements of $V(G)$ (e.g., $b_v$ is the map from $V(G)$ to
  $\mathbb{F}_2$ defined by $b_v(v)=1$ and $b_v(u)=0$ for $u\not =v$). For that,
  set $\varphi (b_u, b_v):=1$ if $u \not =v$ and $u\sim v$; in particular
  $\varphi (b_v, b_v)=0$ for every $v\in V(G)$. Then extend $\varphi$ on $U$ by
  bilinearity. Since the vectors of the basis are isotropic and $\mathbb{F}_2$
  has characteristic two, $\varphi$ is symplectic. By construction, the map
  $v\rightarrow b_v$ is a representation of $G$ in $(U, \varphi)$. Hence $G$ has
  a symplectic dimension.

\item An inner representation of a graph $G$ reduces to a map $f$ from $V(G)$
  into the vector space $[E]^{<\omega}$ of finite subsets of a set $E$ equipped
  with the symmetric difference such that for every two-element subset
  $e:=\{u , v\}$ of $V(G)$, we have $e \in E(G)$ if and only if
  $\card{f(u) \cap f(v)}$ is odd. Suppose that $V(G)= \kappa$ and no subset of
  $V(G)$ of cardinality $\kappa$ is a clique or an independent set. According to
  Ramsey's theorem, $\kappa$ is uncountable. Apply Lemma~\ref{delatsystemlemma}
  to $\mathcal A:= (f(u))_{u\in V(G)}$. 
  Let
  $\mathcal B:=(f(u))_{u\in K}$ be a
subfamily of $\mathcal A$, where $K$ has cardinal $\kappa$, and let
 $R$ be given by this lemma. Since
  $f(u) \cap f(v)= R$ for all every $u, v \in K$, the set $K$ is a
  clique or an independent set depending on the fact that the cardinality of $R$
  is odd or even. Hence, if $G$ has no clique and no independent set of $\kappa$
  vertices, it cannot have an inner representation. A basic example on
  cardinality $\aleph_1$ is provided by the comparability graph $G$ of a
  Sierpinskization of a subchain $A$ of the reals of cardinality $\aleph_1$ with
  an order of type $\omega_1$ on $A$.

\item Let $E:= E(G)$. Let $[E]^{<\omega}$ be the collection of finite subsets of
  $E$; equipped with the symmetric difference $\Delta$, $[E]^{<\omega}$ is a
  vector space over $\mathbb{F}_2$; the one-element subsets of $E$ form a basis;
  the map $\varphi: [E]^{<\omega}\times [E]^{<\omega}\rightarrow \mathbb{F}_2$
  defined by setting $\varphi (X, Y)=1$ if $\vert X\cap Y\vert $ is odd and
  $\varphi (X, Y)=0$ otherwise is a bilinear form for which the one-element
  subsets of $E$ form an orthonormal basis. Hence $\varphi$ is an inner product.
  Let $f:V(G) \rightarrow \powerset (E)$ be defined by setting $f(x):= E_G(x)$
  ($=\{e\in E: x\in e\}$). Since for any pair of distinct vertices
  $x,y\in V(G)$, $\vert E_G(x) \cap E_G(y)\vert=1$ amounts to
  $\varphi(f(x), f(y))=1$, $f$ is an inner representation of $G$. \qedhere
  
\end{enumerate}
\end{proof} 

As noted by 
\citet{delhomme}, the Boolean dimension can be
strictly smaller than the geometric dimension. For an example, if $\kappa$ is an
infinite cardinal, the geometric dimension of $O(\kappa)^{\neg \bot}$ is
$2^\kappa$ while its Boolean dimension is at most $\kappa$. Indeed, from
Theorem~\ref{thm-geometric}, $O(\kappa)^{\neg \bot}$ has a geometric
representation in a vector space $U$.  As for any representation,
Lemma~\ref{lem:onetoone} is still valid; since $O(\kappa)^{\neg \bot}$ has no
duo (for $\kappa\geq 3$) the cardinality of $U$ is at least $2^{\kappa}$, thus
the dimension of the vector space $U$ is $2^{\kappa}$, while
$O(\kappa)^{\neg \bot}$ has a Boolean representation in a set of cardinality
$\kappa$.

 \begin{problem}
   Does every countable graph has an inner dimension? \footnote{Norbert Sauer informed us on january 2022 that the answer is positive}
 \end{problem}

\subsection{Graphs with finite geometric dimension}

\begin{thm}\label{thm:threedimension} If the Boolean dimension of a graph $G$ is
  finite, then it is equal to the inner dimension of $G$ and
either 
\begin{enumerate}
\item the geometric dimension, the symplectic dimension and the Boolean
  dimension of $G$ are equal,

  or \item the geometric dimension and the symplectic dimension of $G$ are equal
  to the Boolean dimension of $G$ minus $1$,

  or \item the geometric dimension and the Boolean dimension of $G$ are equal
  and are strictly less than the symplectic dimension of $G$, in which case the
  difference between these two numbers can be arbitrarily large.
\end{enumerate}
\end{thm}
\begin{proof} The first assertion is obvious.  By definition,
  $\ddgeo(G)\leq \min\{ \ddBool(G), \ddsymp(G)\}$. Apply Theorem~\ref
  {thm:nondegenerate}. Let $k:=\ddgeo(G)$. If $k\not =\ddBool (G)$, then $G$ is
  representable into $H(k)$ and thus in $\mathbb F_2^{k+1}$, hence $(2)$ holds.
  If $k=\ddBool (G)$, then $\ddsymp(G)\geq k$. The examples given in $(a)$ below
  show that the difference $\ddsymp (G)-\ddBool (G)$ can be large.
\end{proof}

We give some examples when the graphs are finite. 
\begin{exs}
\begin{enumerate}[{(a)}]
%

\item $\ddgeo(K^V_X)= \ddBool(K^V_X)=1$ if $\vert X\vert \geq 2$ (in fact they equal
  $0$ if $\vert X\vert \leq 1$) and $\ddsymp(K^V_X)= 2k$ if
  $\vert X\vert \in \{2k, 2k+1\}$.
 
\item
  $\ddgeo(O^{\neg \bot}_{\mathbb F_2}(k))=\ddBool(O^{\neg \bot}_{\mathbb
    F_2}(k))=k$ for $k\geq 2$, and $0$ otherwise.

\item
  $\ddgeo (O^{\neg \bot}_H(k))=\ddsymp(O^{\neg \bot}_H(k))=\ddBool(O^{\neg
    \bot}_H(k))-1=k$ for $k=2m \geq 4$, 
    and \\
  $\ddgeo(O^{\neg \bot}_H(2))=\ddBool(O^{\neg \bot}_H(2))=\ddsymp(O^{\neg
    \bot}_H(2))-1=1$.
\end{enumerate}
\end{exs}
 
These examples are extracted from \cite{bbbp2012}. The paper being unpublished,
we give a hint below. We use the following lemma.

\begin{lem}\label{lem:clique}
  If $G \coloneqq (V,E)$ is a graph for which $\ddsymp (G)=2k\in \NN$, then every
  clique of $G$ has at most $2k+1$ elements.  \end{lem}

This fact is a straightforward consequence of the following claim which appears
equivalently formulated in \cite{lint} as Problem 19O.(i), page 238.
  
\begin{claim} \label{claim:problem} If $\ell +1$ subsets $A_i$, $i <\ell+1$, of
  an $\ell$-element set $A$ have odd size, then there are $i, j < \ell + 1$,
  $i \not = j$ such that $A_i \cap A_j$ has odd size.
\end{claim}

We prove now that the examples satisfy the stated conditions. 

Item (a). The first part is obvious. For the second part, we use
Claim~\ref{claim:problem} and Lemma~\ref{lem:clique}. Indeed, let
$f: V(G)\rightarrow H(2k)$. Composing with the involution $h$ of
$\mathbb F_2^{2k+1}$ we get a representation in $\mathbf 1+H(2k)$, where the
involution $h$ is defined by $h(x) = x+\mathbf{1}$, where
$\mathbf{1}:= (1,1,\ldots,1) \in F_2^{2k+1}$. The image of a clique of $G$
yields subsets of odd size such that the intersection of distinct subsets has
even size. Thus from Claim~\ref {claim:problem} above there are no more than
$2k+1$ such sets.

With that in hand, we prove the desired equality $\ddsymp (K_{X}^{V(G)})= 2k$ if
$\vert X\vert\in \{ 2k, 2k+1\}$.

Indeed, let $X$ be an $n$-element subset of $V(G)$ and let $(x{_i})_{i<n}$ be an
enumeration of $X$. Let $k$ with $n\leq 2k+1$ and
$f: V(G)\rightarrow \mathbb F_2^{2k+1}$ be defined by $f(x)=0$ if
$x\in V(G)\setminus X$ and $f(x):= (b_j)_{j<2k+1}$, where $b_j=1$ for all
$j\not =i$ and $b_i=0$ if $x=x_i$.  Clearly, $f$ is a representation of $G$ in
$H(2k)$, thus $\ddsymp(K_{X}^{V(G)})\leq 2k$.  The reverse inequality follows
from Lemma~\ref{lem:clique}.

Item (b).  If $k=1$, the graph $O^{\neg \bot}_{\mathbb F_2}(k)$ is made of two
isolated vertices, and if $k=2$ the graph is a path on three vertices plus an
isolated vertex, their respective Boolean dimensions are $1$ and $2$, as
claimed. If $k\geq 3$ the result follows from the conclusion of
Lemma~\ref{lem:dimensionphi}.


Item (c) If $k=2$, the graph $O^{\neg \bot}_H(k)$ is made of a clique on three
vertices plus an isolated vertex, hence its Boolean dimension is $1$. If
$k\geq 4$, the equality $\ddgeo(O^{\neg \bot}_H(k))=\ddsymp(O^{\neg \bot}_H(k))$
follows from the conclusion of Lemma~\ref{lem:dimensionphi}.  The number of
edges of $O^{\neg \bot}_H(k)$ and $O^{\neg \bot}_{\mathbb F_2}(k)$ are
different, hence $O^{\neg \bot}_H(k)$ cannot have a Boolean representation in
$(\mathbb F_2^k, |)$. Since it has a representation in $(\mathbb F_2^{k+1}, |)$,
the result follows. \hfill $\Box$

The paper by 
\citet{godsil-royle2001} contains many more
results on the symplectic dimension over $\mathbb F_2$ of finite graphs.

\subsection{Dimension and rank}\label{subsection:dimensionrank}
We compute the symplectic dimension and the geometric dimension of a graph $G$
in terms of its adjacency matrix.

Let $n\in \NN$. Let $A$ be an $n\times n$ symmetric matrix with coefficients in
a field $\mathbb F$. We denote by $\rank_{\mathbb F} (A)$ the rank of $A$
computed over the field $\mathbb F$. The \emph{minrank} of $A$, denoted by
$\minrank_{\mathbb F} (A)$, is the minimum of $\rank_{\mathbb F} (A+D)$, where
$D$ is any diagonal symmetric matrix with coefficients in $\mathbb F$. If
$\mathbb F= \mathbb F_2$, we denote these quantities by $\rank_{2} (A)$ and
$\minrank_{2} (A)$. Let $G:= (V, E)$ be a graph on $n$ vertices. Let
$v_1, \dots, v_n$ be an enumeration of $V$. The \emph{adjacency matrix} of $G$
is the $n\times n$ matrix $A(G):= (a_{i,j}) _{1\leq i, j \leq n}$ such that
$a_{i,j}=1$ if $v_i\sim v_j$ and $a_{i,j}=0$ otherwise.

\begin{thm}\label{thmrank-dimension} If $G$ is a graph on $n$ vertices, then the
  symplectic and the geometric dimensions of $G$ over a field $\mathbb F$ are
  respectively equal to the rank and the minrank of $A(G)$ over $\mathbb F$.
\end{thm}

An $n\times n$ symmetric matrix $B:= (b_{i,j}) _{1\leq i, j\leq n}$ over a field
$\mathbb F$ is \emph{representable} as the matrix of a symmetric bilinear form
$\varphi$ on a vector space $U$ over a field $\mathbb F$ if there exists $n$
vectors $u_1, \dots, u_n$ in $U$, not necessarily distinct, such that
$b_{i,j}= \varphi(u_i, u_j)$ for all ${1\leq i, j \leq n}$.

The matrix $B$ can be represented in $U:=\mathbb F^n$, where
$(u_i)_{1\leq i\leq n}$ is the canonical basis and $\varphi (u_i,u_j)=
b_{i,j}$. According to the following lemma (see Corollary 8.9.2 p. 179 of
\cite{godsil-royle2001bis}), there is a representation in a vector space whose
dimension is the rank of the matrix $B$.
 
\begin{lem} An $n\times n$ symmetric matrix $B$ of rank $r$ has a principal
  $r\times r$ submatrix of full rank.
\end{lem}

The following result shows that this value is optimum.

\begin{lem}\label{rank=dim} The smallest dimension of a vector space in which a
  symmetric matrix $B$ is representable is the rank of $B$.  \end{lem}

\begin{proof}
  It is an immediate consequence of the following facts, whose proofs are a
  simple exercise in linear algebra.

  $1)$ Let $r:= \rank (B)$.  Then $r\leq \dim (U)$ for any vector space $U$ in
  which $B$ is representable.
  Let $\varphi$ be a bilinear form on $U$, and let $u_1, \ldots, u_n$ be $n$
  vectors of $U$ such that $\varphi(u_i,u_j) = b_{ij}$ for all
  $1 \leq i,j \leq n$, where $(b_{i,j})_{1 \leq i,j \leq n} = B$. %
  Let $B(j_1), \dots, B(j_r)$ be $r$ linearly independent column vectors of $B$
  with indices $j_1, \dots, j_r$. We claim that the corresponding vectors
  $u_{j_1}, \dots, u_{j_r}$ are linearly independent in $U$. Suppose that a
  linear combination $ \displaystyle {\sum_{k= 1}^{r}}\lambda_{j_k}u_{j_k}$ is
  zero. Then, for every vector $u\in U$,
  $\varphi ( \displaystyle {\sum_{k= 1}^{r}}\lambda_{j_k}u_{j_k}, u)=0$. This
  rewrites as
  $ \displaystyle {\sum_{k= 1}^{r}}\lambda_{j_k}\varphi (u_{j_k}, u)=0$. In
  particular,
  $ \displaystyle {\sum_{k= 1}^{r}}\lambda_{j_k}\varphi (u_{j_k}, u_i)=0$ for
  every $i= 1,\dots, n$. That is,
  $ \displaystyle {\sum_{k= 1}^{r}}\lambda_{j_k}B_{j_k}=0$. Since these column
  vectors are linearly independent, the $\lambda_{j_k}$'s are zero. This proves
  our claim.

  $2)$ Suppose that $\varphi$ is nondegenerate and $U$ is spanned by the vectors
  $u_{1}, \dots, u_{n}$. Then $r\geq \dim (U)$. The proof follows the same lines
  as above. Let $s:= \dim (U)$. Then, among the $u_j$'s there are $s$ linearly
  independent vectors, say $u_{j_1}, \dots, u_{j_s}$. We claim that the column
  vectors $B(j_1), \dots, B(j_s)$ are linearly independent. Suppose that a
  linear combination $ \displaystyle {\sum_{k= 1}^{s}}\lambda_{k}B_{j_k}$ is
  zero. This yields
  $\displaystyle {\sum_{k= 1}^{s}}\lambda_{k}\varphi (u_{j_k}, u_i)=0$ for every
  $i$, $1\leq i\leq n$, hence
  $\varphi (\displaystyle {\sum_{k= 1}^{s}}\lambda_{k}u_{j_k}, u_i)=0$.  Since
  the $u_i$'s generate $U$, we have
  $\varphi (\displaystyle {\sum_{k= 1}^{s}}\lambda_{k}u_{j_k}, u)=0$ for every
  $u\in U$. Since the form $\varphi$ is nondegenerate,
  $\displaystyle {\sum_{k= 1}^{s}}\lambda_{k}u_{j_k}=0$. Since the vectors
  $u_{j_1}, \dots, u_{j_s}$ are linearly independent, the $\lambda_k$'s are all
  zero.  This proves our claim.

  $3)$ Suppose that $B$ is representable in a vector space $U$ equipped with a
  symmetric bilinear form $\varphi$. Then $B$ is representable in a quotient of
  $U$ equipped with a nondegenerate bilinear form.
\end{proof}  

Theorem~\ref{thmrank-dimension} follows immediately from Lemma~\ref{rank=dim}.   

\begin{rem}Theorem~\ref{thmrank-dimension} for the symplectic dimension of
  graphs over $\mathbb F_2$ is due to 
  \cite{godsil-royle2001}. The minrank over several fields has been intensively
  studied, see \citet{fallat-hogben2007} for a survey.  These
  authors consider the problem of minrank of graphs, and obtain a combinatorial
  description for the minimum rank of trees. In the next section, we only state
  that in case of trees, the Boolean dimension, geometric dimension and the
  minimum rank coincide, thus the formula given in Theorem~\ref{thm:trees} below
  for the Boolean dimension gives yet another combinatorial description for the
  minimum rank of a tree.
\end{rem}

\section{Boolean dimension of trees}

In this section, we show that there is a nice combinatorial interpretation for
the Boolean dimension of trees. We mention first the following result of
Houmem Belkhechine et al. \cite{bbbp2012}.

\begin{lem}
  Let $G \coloneqq (V,E)$ be a graph, with $V\neq \emptyset$. Let $m\in \NN$,
  and let $f\colon V\to \mathbb{F}_2^m$ be a representation of $G$ in the vector
  space $\mathbb{F}_2^m$ equipped with a symmetric bilinear form $\varphi$. Let
  $A\subseteq V$ such that $A\neq \emptyset$. Suppose that for all finite
  $X\subseteq A, X \neq \emptyset$, there exists $v\in V\setminus X$ such that
  $\card{N_G(v)\cap X}$ is odd. Then $\{f(x)\mid x \in A\}$ is linearly
  independent in the vector space $\mathbb{F}_2^m$.
\end{lem}
\begin{proof}
  Let $X$ be a non empty finite subset of $A$.  We claim that
  $\sum_{x\in X}f(x)\not=0$. Indeed, let $v\in V\setminus X$ such that
  $\vert N_G(v)\cap X\vert$ is odd. We have
  $\varphi(\sum_{x\in X} f(x), f(v))= \sum_{x\in X}\varphi(f(x), f(v))$. This
  sum is equal to $\vert N_G(v)\cap X\vert$ modulo $2$. Thus
  $\sum_{x\in X} f(x)\not =0$ as claimed. Since this holds for every finite
  subset $X$ of $A$, the conclusion follows.
\end{proof} 

This suggests the following definition.

\begin{definition}[\citep{bbbp2012}] Let $G \coloneqq (V,E)$
  be a graph.  A set $A\subset V$ is called {\em independent $\pmod{2}$} if for
  all finite $X\subseteq A, X \neq \emptyset$, there exists $v \in V\setminus X$
  such that $\card{N_G(v)\cap X}$ is odd, otherwise $A$ is said to be dependent
  $\pmod{2}$. Let $\ind(G)$ be the maximum size of an independent set $\pmod{2}$
  in $G$. \textbf{From now, we omit $\pmod{2}$ unless it is necessary to talk
    about independence in the graph theoretic sense.}
\end{definition}

\begin{cor}\label{cor:independent}
  For every graph $G$, we have $\ind(G)\leq \ddgeo(G)$.
\end{cor} 

\begin{problem} Does the equality hold?
\end{problem} 

Note that the independent sets $\pmod {2}$ of a graph do not form a matroid in
general. Indeed, let $G$ be made of six vertices, three, say $\{a,b,c\}$ forming
a clique, the three others, say $a',b',c'$ being respectively connected to $a,b$
and $c$. Then $\{a', a, b,c\}$ is independent $\pmod {2}$, hence
$4\leq \ind(G)$. Also, $\{a', b', c'\}$ is independent $\pmod {2}$ but cannot be
extended to a larger independent set $\pmod {2}$. Since $G$ is the Boolean sum
of a $3$-vertex clique and three edges, $\ddBool(G)\leq 4$. Finally,
$\ind(G)= \ddgeo(G)= \ddBool(G)=4$.
 
From Corollary~\ref{cor:independent} above, we deduce the following result.

\begin{thm}\label{thm:paths}
  The Boolean dimension of a path on $n$ vertices ($n\in \NN, n > 0$) is
  $n-1$. Every other $n$-vertex graph, with $n \geq 2$, has dimension at most
  $n-2$.
\end{thm}

\begin{proof}Let $P_n$ be the path on $\{0, \dots, n-1\}$, whose edges are pairs
  $\{i, i+1\}$, with $i<n-1$. Suppose $n \geq 2$. Since $P_n$ is the Boolean sum
  of its edges, $\ddBool (P_n)\leq n-1$. Let $A:= \{0, \dots, n-2\}$. Then $A$
  is independent $\pmod{2}$. Indeed, let $X$ be a nonempty subset of $A$ and $x$
  be its largest element, then the vertex $v:=x+1$ is such that
  $\card{N_{P_n}(v)\cap X}=1$. Thus $\ind (P_n)\geq n-1$. From the inequalities
  $n-1\leq \ind(P_n)\leq \ddgeo(P_n)\leq \ddBool (P_n)\leq n-1$, the fact that
  the dimension of $P_n$ is $n-1$ follows.

  Now we prove that if the Boolean dimension of a graph $G$ on $n$ vertices is
  $n-1$, then $G$ is a path.  Observe first that $G$ is connected. Otherwise,
  $G$ is the direct sum $G'\oplus G''$ of two non trivial graphs $G'$ and $G''$
  with respectively $n'$ and $n''$ vertices.  As it is immediate to see,
  $\ddBool(G)=\ddBool (G'\oplus G'')\leq \ddBool(G')+\ddBool(G'')\leq
  n'-1+n''-1=n-2$.  Next we observe that $G$ cannot be a cycle. Indeed, an easy
  induction shows that cycles on $n$ vertices have dimension at most
  $n-2$. Indeed, the cycle $C_3$ is a clique thus has dimension $1$. For
  $n\geq 4$, the cycle $C_n$ on $n$ vertices $\{0, \dots, n-1\}$ is the Boolean
  sum of the cycle on the first $n-1$ vertices and the $3$-vertex cycle on
  $\{0, n-2, n-1\}$, thus its dimension is at most $n-2$ (in fact it is equal to
  $n-2$; this is obvious for $C_4$ while for $n\geq 5$, its dimension is at
  least $n-2$ since it contains a path on $n-1$ vertices). Next, we check that
  if $G$ has no more than four vertices, then it is a path. For the final step,
  we argue by induction, but we need a notation.  Let $G:= (V, E)$ be a graph
  and $x \in V$. Let $G_{-x}$ be the subgraph of $G$ induced by
  $V\setminus \{x\}$. Let $G^{x} := (G_{- x}) \dot{+} K_G(x)$, where
  $K_G(x):=K_{N_{G}(x)}^{V\setminus \{x\}}$. Let $\dot{G}^x$ be the graph
  obtained by adding to $G^x$ the vertex $x$ as an isolated vertex. In simpler
  terms, we obtain $G^x$ by deleting from $G$ the vertex $x$ and by adding, via
  the Boolean sum, all edges between vertices of $N_{G}(x)$. For an example, if
  $G$ is a path then $G^x$ is a path on $V\setminus \{x\}$.
  
  \begin{claim}\label{claim:exponent} If $V$ is finite then
    $\vert \ddBool (G)- \ddBool (G^{x})\vert \leq 1$.
  \end{claim}

  \noindent{\bf Proof of Claim~\ref{claim:exponent}}
  Note that $\dot{G}^x \dot{+} K_{N_{G}(x)\cup \{x\}}^V= G$ and
  $G \dot{+} K_{N_{G}(x)\cup \{x\}}^V= \dot {G}^x$.  Thus
  $\vert \ddBool (G)- \ddBool(\dot {G}^{x})\vert \leq \ddBool (K_{N_{G}(x)\cup
    \{x\}}^V)$.  Since $K_{N_{G}(x)\cup \{x\}}$ is a clique, its Boolean
  dimension is $1$; and since $\dot{G}^x$ and $G^x$ differ by an isolated
  vertex, they have the same Boolean dimension.  The claimed inequality follows.

  Now, let $G$ be our graph on $n$ vertices such that $\ddBool(G)=n-1$. Suppose
  that every graph $G'$ on $n'$ vertices, $n'<n$, is a path whenever
  $\ddBool(G')=n'-1$.

  \begin{claim}\label{claim:path1} $G^x$ is a path for every $x\in V(G)$.
  \end{claim}
  Indeed, since $G^x$ has $n-1$ vertices, $\ddBool(G^x)\leq n-2$; since
  $\ddBool(G)=n-1$, the claim above ensures that $\ddBool(G^x)=n-2$. The
  conclusion follows for the hypothesis on graphs with $n-1$ vertices.

  \begin{claim}\label{claim:path2}
    Let $x, y\in V(G)$ with $x\not =y$. If $G^{x}$ and $G^{y}$ are two paths
    $P_x$ and $P_y$, then $d_G(x), d_G(y)\leq 2$ if $\{x,y\} \not \in E(G)$, and
    $d_G(x), d_G(y)\leq 3$ otherwise.
  \end{claim}
  \noindent {\bf Proof of Claim~\ref{claim:path2}}
  We have
  $G^{x}\dot{+}G^{y}=G \dot{+} K_{N_{G}(x)\cup \{x\}}^V\dot{+}G \dot{+}
  K_{N_{G}(y)\cup \{y\}}^V=K_{N_{G}(x)\cup \{x\}}^V\dot{+}K_{N_{G}(y)\cup
    \{y\}}^V$. Since $G^{x}$ and $G^{y}$ are two paths $P_x$ and $P_y$,  
  $P_x\dot{+} P_y= K_{N_{G}(x)\cup \{x\}}^V\dot{+}K_{N_{G}(y)\cup \{y\}}^V$. 
  We
  have $d_{P_x\dot{+} P_y}(x)\leq 2$, hence for the Boolean sum
  $E_G(x)\dot{+} E_G(y)$ of stars $E_G(x)$ and $E_G(y)$, we have
  $d_{E_G(x)\dot{+} E_G(y)}(x)\leq 2$. The conclusion of the claim follows.

  Now let $x\in V$. Since $d_G(x)\leq 3$ and $n\geq 5$ there is some vertex $y$
  not linked to $x$ by an edge. Hence by Claim~\ref{claim:path2},
  $d_G(x)\leq 2$. From this follows that $G$ is a direct sum of paths and
  cycles.

  Since $G$ must be connected and cannot be a cycle, $G$ is a path.
\end{proof}

We thank 
\cite{bondy} for suggesting this result several years
ago. In fact, it is a consequence of previous results about geometric dimension
of graphs, obtained for general fields 
\citep{bento-lealduarte,rheinboldt-shepherd}.

We go from paths to trees as follows. 

\begin{definition}
  Let $T \coloneqq (V,E)$ be a tree. A {\em star decomposition} $\Sigma$ of $T$
  is a family $\{S_1,\ldots,S_k\}$ of subtrees of $T$ such that each $S_i$ is
  isomorphic to $K_{1,m}$ (a star) for some $m \geq 1$, the stars are mutually
  edge-disjoint, and each edge of $T$ is an edge of some $S_i$. For a star
  decomposition $\Sigma$, let $t(\Sigma)$ be the number of trivial stars in
  $\Sigma$ (stars that are isomorphic to $K_{1,1}$), and let $s(\Sigma)$ be the
  number of nontrivial stars in $\Sigma$ (stars that are isomorphic to $K_{1,m}$
  for some $m > 1$). We define the parameter
  $m(T) \coloneqq \min_{\Sigma} \{t(\Sigma) + 2s(\Sigma)\}$ over all star
  decompositions $\Sigma$ of $T$. A star decomposition $\Sigma$ of $T$ for which
  $t(\Sigma) + 2s(\Sigma) = m(T)$ is called an {\em optimal star decomposition}
  of $T$.
\end{definition}

The Boolean dimension of a graph counts the minimum number of cliques needed to
obtain this graph as a Boolean sum.  If $\Sigma:=\{S_1,\ldots,S_k\}$ is a star
decomposition of a tree $T$, one has
$\ddBool (T) \leq \displaystyle {\sum_{i= 1}^{k}}\ddBool (S_i)$. Since
$\ddBool (S_i)= 1$ if $S_i$ is a trivial star, and $\ddBool (S_i)= 2$ otherwise
(note that if $S_i= K_{1, m}$, it is the Boolean sum of a clique on $m+1$
vertices and a clique on a subset of $m$ vertices), hence we have
$\displaystyle {\sum_{i= 1}^{k}}\ddBool (S_i)= t(\Sigma)+2s(\Sigma)$, hence
$\ddBool(T)\leq t(\Sigma)+2s(\Sigma)$. The inequality $\ddBool(T) \leq m(T)$
follows.

Here is our result. 

\begin{thm}\label{thm:trees}
  For all trees $T$, we have $\ind(T) = \ddBool(T) = m(T)$.
\end{thm}

We introduce the following definition.

\begin{definition}
  A {\em cherry} in a tree $T$ is a maximal subtree $S$ isomorphic to
  $K_{1,m}$ for some $m > 1$ that contains $m$ end vertices of $T$. We
  refer to a cherry with $m$ edges as an $m$-cherry.
\end{definition}

\begin{prop}\label{prop:cherry}
  Let $T \coloneqq (V,E)$ be a tree that contains a cherry. If all
  proper subtrees $T^\prime$ of $T$ satisfy $\ind(T^\prime) = m(T^\prime)$,
  then $\ind (T) = m(T)$.
\end{prop}

\begin{proof}
  Let $x\in V$ be the center of a $k$-cherry in $T$, with
  $N_T(x) = \{u_1,\ldots, u_k, w_1,\ldots,w_{\ell}\}$, where $d_T(u_i) = 1$ for
  all $i$, and $d_T(w_i) > 1$ for all $i$.  For each $i = 1 \text{ to } \ell$,
  let $T_i$ be the maximal subtree that contains $w_i$ but does not contain $x$.

  First, we show that any optimal star decomposition of $T$ in which $x$ is not
  the center of a nontrivial star can be transformed into an optimal star
  decomposition in which $x$ is the center of a nontrivial star. Consider an
  optimal star decomposition $\Sigma$ in which $x$ is not the center of a
  nontrivial star. Therefore, edges $xu_i$ are trivial stars of $\Sigma$. Now if
  $k > 2$ or if there is a trivial star $xw_i$ in $\Sigma$, then we could have
  improved $t(\Sigma) + 2 s(\Sigma)$ by replacing all trivial stars containing
  $x$ by their union, which is a star centered at $x$. Hence, assume that $k=2$
  and each $w_i$ is the center of a nontrivial star $S_i$, which contains the
  edge $xw_i$. Now replace each $S_i$ by $S_i^\prime \coloneqq S_i - xw_i$, and
  add a new star centered at $x$ with edge set
  $\{xw_1,\ldots,xw_{\ell},xu_1,xu_2\}$. The new decomposition is also optimal.

  Now consider an optimal star decomposition $\Sigma$ in which $x$ is
  the center of a nontrivial star. The induced decompositions on $T_i$ are all
  optimal since $\Sigma$ is optimal. For each
  $i\in \{1,\ldots, {\ell}\}$, let $A_i$ be a maximum size independent set in
  $T_i$. Hence $\card{A_i} = \ind(T_i) = m(T_i)$ for all $i\geq 1$, and
  $m(T) = 2+\sum_i m(T_i) = 2+\sum_i \ind(T_i)$. We show that
  $A \coloneqq \{x,u_1\} \cup (\cup_i A_i)$ is a maximum size
  independent set in $T$.

  Consider a non-empty set $X\subseteq A$. We show that there exists
  $v\in V\setminus X$ such that $\card{N_T(v) \cap X}$ is odd. %
  If $x \in X$, we have $N_T(u_2)\cap X = \{x\}$. If $X = \{u_1\}$, we have
  $N_T(x)\cap X = \{u_1\}$. So suppose $x \not \in X$ and $X \neq \{u_1\}$. %
  Let $B_i \coloneqq X\cap V(T_i)$ for $i \in \{1,\ldots, {\ell}\}$. Since $B_i$
  is nonempty for some $i$, and $x \not \in X$, we find
  $v \in V(T_i)\setminus B_i$ such that $\card{N_{T_i}(v)\cap B_i}$ is odd. Now
  $\card{N_{T}(v)\cap X}$ is odd since $x \not \in X$ and $v$ is not adjacent to
  $u_1$. Moreover, $\card{A} = m(T)$.
\end{proof}

\begin{prop}\label{prop:deg2}
  Let $T \coloneqq (V,E)$ be a tree that contains a vertex $y$ of degree
  2 adjacent to a vertex $z$ of degree 1. If $\ind(T-z) = m(T-z)$, then
  $\ind(T) = m(T)$.
\end{prop}

\begin{proof}
  First, we show that $m(T) = m(T-z)+1$. If there is an optimal star
  decomposition of $T-z-y$ in which some vertex $x$ is the center of a star,
  then $m(T-z) = m(T-z-y)$ and $m(T) = m(T-z)+1$, else $m(T-z) = m(T-z-y)+1$ and
  $m(T) = m(T-z-y)+2$.
    
  Now we consider a maximum sized independent set $A^\prime$ in $T-z$. We
  have $\card{A^\prime} = \ind(T-z) = m(T-z)$. We define
  $A \coloneqq A^\prime \cup \{y\}$ if $y\not \in A^\prime$; and
  $A \coloneqq A^\prime \cup \{z\}$ if $y \in A^\prime$. We show that
  $A$ is independent in $T$.

  \
  
  \noindent {\em Case 1: $y \not \in A^\prime$, hence $y \in A$ and
    $z\not \in A$.}  Let $B\subseteq A, B \neq \emptyset$.

  If $y \in B$, then $\card{N_T(z)\cap B}$ is odd.

  If $y \not \in B$, then $B \subseteq A^\prime$, hence there exists
  $v\in V(T-z)$ such that $\card{N_{T-z}(v)\cap B}$ is odd, and
  $\card{N_{T}(v)\cap B}$ is odd.

  \
  
  \noindent {\em Case 2: $y \in A^\prime$, hence $z \in A$.} Let
  $B \subseteq A, B \neq \emptyset$.

  If $z\not \in B$, then $B\subseteq A^\prime$. Find
  $v\in V(T-z)\setminus B$ such that $\card{N_{T-z}(v)\cap B}$ is
  odd. Hence $\card{N_{T}(v)\cap B}$ is odd.

  Now suppose that $z\in B$. If $B = \{z\}$, then $\card{N_T(y)\cap B}$ is
  odd. Otherwise, consider $B \setminus \{z\}$, which is a subset of
  $A^\prime$. Find $v\in V(T-z) \setminus (B \setminus \{z\}) $ such that
  $\card{N_{T-z}(v)\cap (B\setminus \{z\})}$ is odd. If $v \neq y$, then
  $\card{N_{T}(v)\cap B}$ is odd. %
  If $v=y$ and $x \in N_T(y)\setminus \{z\}$, then $\card{N_{T}(v)\cap B}$ is
  even and $x\in B$. In this case, let
  $B^\prime \coloneqq (B\setminus \{z\})\cup \{y\}$. This is a subset of
  $A^\prime$. Find $u \in V(T-z)\setminus B^\prime$ such that
  $\card{N_{T-z}(u)\cap B^\prime}$ is odd. Since $B^\prime$ contains $x$ and
  $y$, we conclude that $u$ is not adjacent to any of $y$ and $z$, hence
  $\card{N_T(u)\cap B}$ is odd.

  Thus we have shown that $A$ is independent. We have
  $\ind(T) \geq \card{A} = \card{A^\prime} + 1 = m(T-z)+1 = m(T)$. Since
  $\ind(T)$ cannot be more than $m(T)$, we have $\ind(T) = m(T)$.
\end{proof}

\begin{proof}[Proof of Theorem~\ref{thm:trees}] If a tree $T$ has two
  vertices, then $\ind(T) = m(T) = 1$.  Each tree with at least $3$
  vertices contains a cherry or a vertex of degree $2$ adjacent to a
  vertex of degree $1$. (This is seen by considering the second-to-last
  vertex of a longest path in $T$.)  Now,  induction on the number of
  vertices, using Propositions~\ref{prop:cherry} and~\ref{prop:deg2},
  implies the result.
\end{proof}

\section{Inversion index of a tournament and Boolean
  dimension}\label{section:inversionindex}

\subsection{Inversion index of a tournament} Let $T$ be a tournament.  Let
$V(T)$ be its vertex set and $A(T)$ be its arc set. An {\it inversion} of an arc
$a := (x,y)\in A(T)$ consists to replace the arc $a$ by $a^{\star} := (y,x)$ in
$A(T)$. For a subset $X \subseteq V(T)$, let $Inv(T, X)$ be the tournament
obtained from $T$ after reversing all arcs $(x,y) \in A(T)\cap (X\times X)$. For
example, $Inv(T, V)$ is $T^*$, the {\it dual} of $T$. For a finite sequence
$(X_{i})_{ i < m}$ of subsets of $V(T)$, let $Inv( T, (X_{i})_{i< m})$ be the
tournament obtained from $T$ by reversing successively all the arcs in each of
the subsets $X_i$, $i<m$, that is, the tournament equal to $T$ if $m=0$ and to
$Inv(Inv(T, (X_{i})_{i< m-1}), X_{m-1})$ if $m\geq 1$.  Said differently, an arc
$(x,y)\in A(T)$ is reversed if and only if the number of indices $i$ such that
$\{x,y\} \subseteq X_i$ is odd. The {\it inversion index} of $T$, denoted by
$i(T)$, is the least integer $m$ such that there is a sequence $(X_{i})_{i< m}$
of subsets of $V(T)$ for which $Inv( T, (X_{i})_{i< m})$ is acyclic.

In the sequel, we consider tournaments for which this index is finite. In full
generality, the inversion index of a tournament $T$ can be defined as the least
cardinal $\kappa$ such the Boolean sum of $T$ and a graph of Boolean dimension
$\kappa$ is acyclic. The case $\kappa$ finite is stated in
Lemma~\ref{lem:inversionboolean} below. We leave tournaments with infinite
inversion index to further studies.

The motivation for the notion of inversion index originates in the study of
critical tournaments. Indeed, the critical tournaments of 
 \citet {ST} can be easily defined from acyclic tournaments by means of one or two
inversions whereas the $(-1)$-critical tournaments, characterized in 
\citet{HIJ2}, can be defined by means of two, three or four inversions 
\citet{belkhechine}. Another interest comes from the point of view of logic.

Results about the inversion index originate in the thesis of 
\citet{belkhechine}. Some results have been announced in \citet{bbbp2010}; they
have been presented at several conferences by the first author and included in a
circulating manuscript \cite{bbbp2012}. The lack of answer for some basic
questions is responsible for the delay of publication.

The inversion index is a variant of the \emph{Slater index}: the least number of
arcs of a tournament which have to be reversed in order to get an acyclic
tournament \citep{SL}. The complexity of the computation of the Slater index
was raised by 
\citet{BT}.  \citet*{alon}  and
independently 
\citet*{charbit} showed 
that the
problem is NP-hard. An extension of the inversion index to oriented graphs is
studied in \citet{bang-jensen-all:hal-02918958}.
\begin{problem}
  Is the computation of the inversion index NP-hard?
\end{problem}

\begin{question}
Are there tournaments of arbitrarily large inversion index?
\end{question}

This last question has a positive answer. There are two reasons, the first one
is counting, the second one, easier, is based on the notion of
well-quasi-ordering.

For $n \in \mathbb{N}$, let $i(n)$ be the maximum of the inversion index of
tournaments on $n$ vertices.  We have $i(n)=0$ for $n\leq 2$, $i(3)=i(4)=1$,
$i(5)=i(6)=2$. For larger $n$ a counting argument \citet{belkhechine, bbbp2010,
  bbbp2012} yields the following result.

\begin{thm} \label{THM 0} $\frac{n-1}{2} - \log_{2} n \leq i(n) \leq n-4$ for
  all integer $n \geq 6$.
\end{thm}

It is quite possible that $i(n)\geq \lfloor{\frac{n-1}{2}}\rfloor$, due to the
path of strong connectivity (it is not even known if reverse inequality holds).

The \emph{path of strong connectivity on $n$} vertices is the tournament $T_n$
defined on $ \mathbb {N} _ {<n}:= \{0, \dots, n-1\}$ whose arcs are all pairs
$(i, i+1)$ and $(j, i)$ such that $i+1<n$ and $i<j<n$.
 
 \vspace{-1.7cm}
 
\begin{figure}[htb]
\begin{center}
\includegraphics[width=5cm]{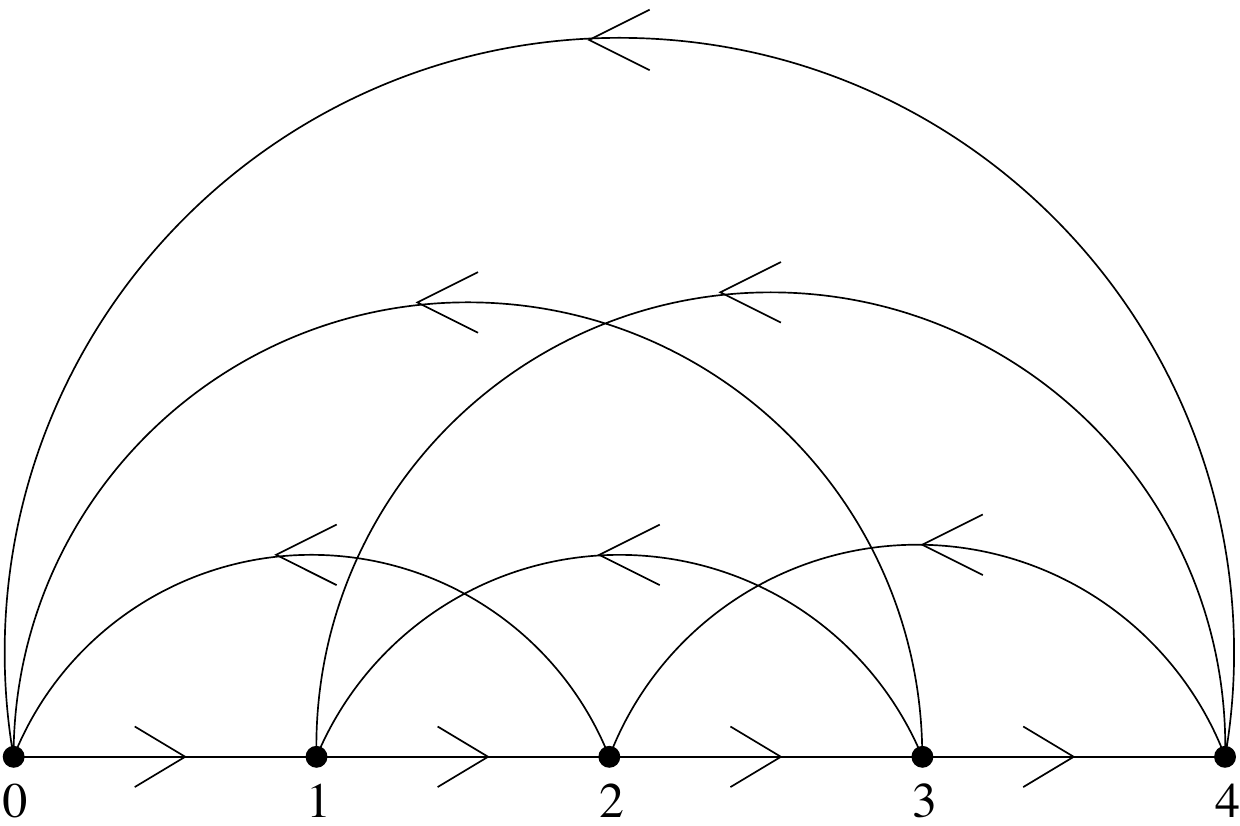}
\end{center}
\vspace{-1.7cm}
\caption{\text{Path of strong connectivity on $5$ vertices}}
\label{Pathtournament5}
\end{figure}

\begin{question}
  Is the inversion index of a path of strong connectivity on $n$ vertices equal
  to $\lfloor{\frac{n-1}{2}}\rfloor$?
\end{question}

\subsection{Well-quasi-ordering} Basic notions of the theory of relations apply
to the study of the inversion index. These notions include the quasi order of
embeddability, the hereditary classes and their bounds, and the notion of
well-quasi-order. For those, we refer to 
\citet{fraisse}.

Let $\mathcal {I}_{m}^{< \omega}$ be the class of finite tournaments $T$ whose
inversion index is at most $m$. This is a hereditary class in the sense that if
$T\in \mathcal {I}_{m}^{< \omega}$ and $T'$ is embeddable into $T$ then
$T'\in \mathcal {I}_{m}^{< \omega}$. It can be characterized by obstructions or
bounds.  A \emph{bound} is a tournament not in $\mathcal {I}_{m}^{< \omega}$
such that all proper subtournaments are in $\mathcal {I}_{m}^{< \omega}$. We may
note that the inversion index of every bound of $\mathcal {I}_{m}^{< \omega}$ is
at least $m+1$. Hence, the fact that $\mathcal {I}_{m}^{< \omega}$ is distinct
of the class of all finite tournaments provides tournaments of inversion index
larger than $m$. This fact relies on the notion of well-quasi-ordering.

A poset $P$ is \emph{well-quasi-ordered} if every sequence of elements of $P$
contains an increasing subsequence.
\begin{thm}\label{latka} 
The class of all finite tournaments is not well-quasi-ordered by embeddability.
\end{thm}

This is a well known fact. As indicated by a referee, it has been mentioned by
several authors. See e.g., 
\cite{latka} for a much stronger version of
Theorem~\ref{latka} and also subsection 3.1 of \citet{cherlin}. For the
convenience of the reader we give a proof.

\begin{proof} Let $T_n$ be the path of strong connectivity on
  $\{0, \dots, n-1\}$ as defined above. Let $C_n$ be the tournament obtained
  from $T_n$ by reversing the arc $(n-1, 0)$. We claim that for $n\geq 7$, the
  $C_n$'s form an antichain. Indeed, to $C_n$ we may associate the $3$-uniform
  hypergraph $H_n$ on $\{0, \dots, n-1\}$ whose $3$-element hyperedges are the
  $3$-element cycles of $C_n$.  An embedding from some $C_n$ to another $C_m$,
  $m > n$, induces an embedding from $H_n$ to $H_m$. To see that such an
  embedding cannot exist, observe first that the vertices $0$ and $n-1$ belong
  to exactly $n-3$ hyperedges, and the vertices $1$ and $n-2$ belong to exactly
  two hyperedges, the other vertices to three hyperedges, hence an embedding $h$
  will send $\{0, n-1\}$ on $\{0, m-1\}$.  The preservation of the arc
  $(0,n-1)$ imposes $h(0)=0$ and $h(n-1)= m-1$. Then, the preservation of the
  arcs $(i, i+1)$ yields a contradiction since $n<m$.
\end{proof}

\begin{thm}\citet{bbbp2010}\label{thm:wqo}
  For each $m\in \NN$, the class $\mathcal {I}_{m}^{< \omega}$ is well-quasi-ordered.
\end{thm}

\begin{proof}
  The class $\mathcal {L}_{m}^{< \omega}$ made of a finite linear order $L$ with
  $m$ unary predicates $U_1, \dots, U_{m}$ (alias $m$ distinguished subsets) and
  ordered by embeddability is well-quasi-ordered. This is a straightforward
  consequence of Higman's theorem on words 
\citep[see][]{higman1952}, in fact, an
  equivalent statement.  Higman's result asserts that the collection of words
  on a finite alphabet, ordered by the subword ordering, is well-quasi-ordered. Since members of $\mathcal {L}_{m}^{< \omega}$ can be coded by words
  on an alphabet with $2^{m}$ elements, the class $\mathcal {L}_{m}^{< \omega}$
  is well-quasi-ordered. The map associating to each $(L, U_1, \dots, U_{m} )$
  the Boolean sum $L\dot {+} U_1\dots \dot {+} U_{m}$ preserves the
  embeddability relation, hence the range of that map is well-quasi-ordered.
  This range being equal to $\mathcal {I}_{m}^{< \omega}$, this later class is
  well-quasi-ordered.\end{proof}

\begin{cor}There are finite tournaments with arbitrarily large inversion index.  
\end{cor}

We have the following result concerning the bounds.

\begin{thm}\citet{bbbp2010}
The  class $\mathcal {I}_{m}^{< \omega}$ has only finitely many bounds. 
\end{thm}

\begin{proof} 
  From the proof of Theorem~\ref{thm:wqo}, the class
  $\mathcal {I}_{m, 1}^{< \omega}$ made of tournaments of
  $\mathcal {I}_{m}^{<\omega}$, with one unary predicate added, is
  well-quasi-ordered. According to an adaptation of Proposition 2.2 of \citet
  {pouzet2} translated in this case, $\mathcal {I}_{m}^{< \omega}$ has finitely
  many bounds.
\end{proof}

We thank the referee for observing that the well-quasi-ordering of
$\mathcal {I}_{m, 1}^{< \omega}$ suffices to yield the finiteness of the bounds
of $\mathcal {I}_{m}^{< \omega}$.

\begin{question} What is the maximum of the cardinality of bounds of
  $\mathcal {I}_{m}^{< \omega}$?
\end{question}

\begin{rem} It must be observed that the collection of graphs with geometric
  dimension at most $m$ over a fixed finite field has finitely many bounds and
  an upper bound on their cardinality is given in \citet{ding-kotlov}. How the
  cardinality of these bounds relate to the cardinality of bounds of
  $\mathcal {I}_{m}^{< \omega}$ is not known.
\end{rem}

\subsection{Boolean dimension and concrete examples of tournaments with large
  inversion index}

Let $C_3.{\underline n}$ be the sum of copies of the $3$-cycle $C_3$ indexed by
the $n$-element acyclic tournament
$\underline n:=(\{0, \dots , n-1\},\{(i,j)\mid 0 \leq i < j \leq n-1\})$ with
$0<\cdots <n-1$.

\begin{thm}\label{thm:threecycle}
  The inversion index of the sum $C_3.{\underline n}$ of $3$-cycles over an
  $n$-element acyclic tournament is $n$.
\end{thm}
\begin{figure}[htb]
\begin{center}
\includegraphics[width=6cm]{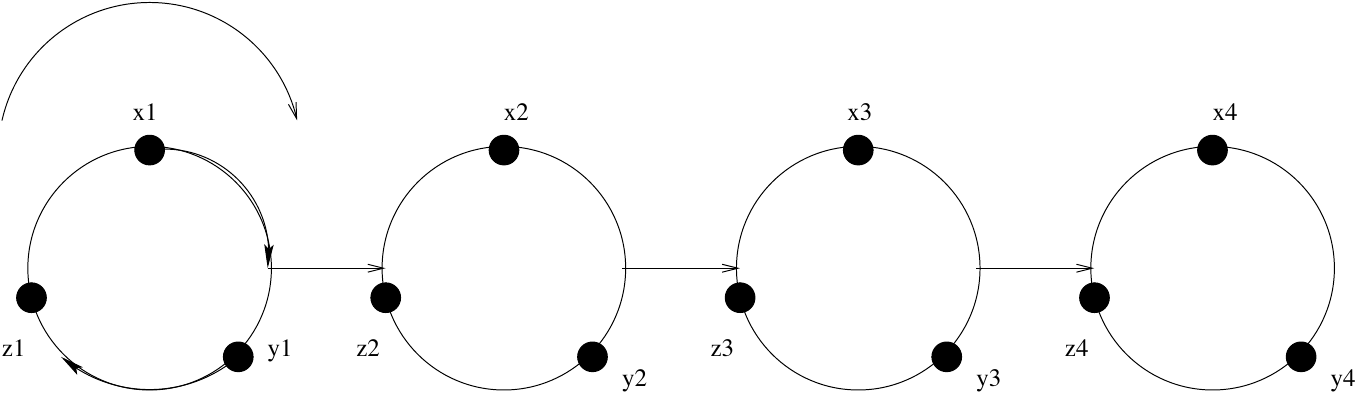}
\end{center}
\caption{\text{$C_3.{\underline 4}$}}
\label{Pathtournament5}
\end{figure}

No elementary proof is known.  The proof we present relies on the notion of
Boolean sum of graphs.

According to the definition of Boolean sum, we have the following result
immediately.

%
 
 \begin{lem}\label{lem:inversionboolean}
   The inversion index of a tournament $T$ is equal to the least integer $k$
   such that the Boolean sum $T\dot{+} G$ of $T$ with a graph $G$ of Boolean
   dimension $k$ is an acyclic tournament.
 \end{lem}

\noindent {\bf Proof of  Theorem~\ref{thm:threecycle}.}
Let $T:=C_3.{\underline n}$, $V:=V(T)$ and $r:=i(T)$. Clearly $r\leq n$.
Conversely, let $H$ be a graph with vertex set $V$ such that $L:=T\dot{+}H$ is
an acyclic tournament and $\ddBool(H)=r$. Let $U:=(\mathbb F_2)^r$ equipped with
the ordinary scalar product $\vert$ and $f: V \rightarrow U$ be a representation
of $H$.

\begin{claim} For each $i\in \{0, \dots, n-1\}$, we may enumerate the vertices
  of $\{0,1,2\}\times \{i\}$ into $x_i,y_i, z_i$ in such a way that
  $(x_i,y_i),(y_i,z_i), (z_i,x_i)$ are arcs of $T$, $(f(x_i) \vert f(z_i))=1$
  and $(f(x_i) \vert f(y_i))=0$.
\end{claim}
\begin{claim}\label{claim:linearlyindt} The set $\{f(x_i): i<n\}$ is linearly
  independent in $U$.
\end{claim}
\noindent {\bf Proof of Claim \ref{claim:linearlyindt}.} This amounts to prove
that $\sum_{i\in I}f(x_i)\not =0$ for every non-empty subset $I$ of
$\{0,\dots, n-1\}$. Let $I$ be such a subset. Let $m\in \{0, \dots, n-1\}$ such
that $x_m$ is the largest element of $\{x_i: i\in I\}$ in the acyclic tournament
$L$. 

\begin{subclaim}\label{subclaim:indtcycle}
  $(f(x_i) \vert f(z_m))=(f(x_i) \vert f(y_m))$ for each
  $i\in I\setminus \{m\}$.
\end{subclaim}
\noindent {\bf Proof of Subclaim~\ref{subclaim:indtcycle}.} By construction, we
have $x_{m}<_L z_m$ and $x_{m}<_L y_m$, hence by transitivity $x_{i}<_L z_m$ and
$x_{i}<_L y_m$.  If $i<m$ in the natural order then, by definition of $T$,
$(x_{i}, z_m)\in A(T)$ and $(x_{i}, y_m)\in A(T)$, thus
$(f(x_i) \vert f(z_m))=0=(f(x_i) \vert f(y_m))$, whereas if $i>m$ in the natural
order, then $(z_m, x_{i})\in A(T)$ and $(y_m, x_{i})\in A(T)$, thus
$(f(x_i) \vert f(z_m))=1=(f(x_i) \vert f(y_m))$, proving the subclaim. \hfill $\Box$

Since $(f(x_{m}) \vert f(z_{m}))=1$ and $(f(x_{m}) \vert f(y_{m}))=0$, it
follows that
$\sum_{i\in I}(f(x_i)\vert f(z_{m}))\not =\sum_{i\in I}(f(x_i)\vert
f(y_{m}))$. That is
$((\sum_{i\in I}f(x_i))\vert f(z_{m}))\not =((\sum_{i\in I}f(x_i))\vert
f(y_{m}))$. Thus the sum $\sum_{i\in I}f(x_i)\not =0$ as claimed.  \hfill $\Box$

We have  $n\leq r$. This proves the theorem. \hfill $\Box$

\acknowledgements
\label{sec:ack}
The authors thank a referee of this paper for his extremely careful and
detailed suggestions which significantly improved the presentation of this paper
and also the content (e.g., Theorem~\ref{thm-geometric}).  The third author
would like to thank Institut Camille Jordan, Universit\'e Claude Bernard Lyon 1
for hospitality and support. Support from CAPES Brazil (Processo: :
88887.364676/2019-00) and Labex MILYON: ANR-10-LABX-0070 are gratefully
acknowledged.  The first author is pleased to thank Uri Avraham, Adrian Bondy
and Christian Delhomm\'e for their
contribution.

\nocite{*}
\bibliographystyle{abbrvnat}
\bibliography{pst-dmtcs}
\label{sec:biblio}

\end{document}